\documentclass[12pt,a4paper,twoside]{article}

\usepackage{amsfonts,amsmath,amssymb,amsthm,latexsym}
\usepackage{times,mathptm}
\usepackage{handelingen}

\newtheorem{proposition}{Proposition}[section]
\newtheorem{lemma}[proposition]{Lemma}
\newtheorem{cor}[proposition]{Corollary}
\newtheorem{theorem}[proposition]{Theorem}

\theoremstyle{definition}

\theoremstyle{remark}
\newtheorem{remark}[proposition]{Remark}

\newcommand{\selabel}[1]{\label{se:#1}}

\newcommand{\ncm}{\newcommand}

\newcommand{\ch}{\chi}
\newcommand{\Lam}{\Lambda}
\newcommand{\bn}{\begin}
\newcommand{\al}{\alpha}
\newcommand{\dw}{\downarrow}
\newcommand{\uw}{\uparrow}
\newcommand{\mtr}{\mathrm}
\newcommand{\ot}{\otimes}

\ncm{\beq}{\begin{equation}}
\ncm{\eeq}{\end{equation}}

\def\C{\mathbb{C}\,}
\def\Z{\mathbb{Z}\,}
\def\N{\mathbb{N}\,}

\def\id{\mbox{\rm id}}

\def\deg{\mbox{\rm deg}\,}

\def\into{\hookrightarrow}
\def\to{\rightarrow}

\def\End{\mbox{\rm End}\,}

\def\Hom{\mbox{\rm Hom}\,}
\def\Ind{\mbox{\rm Ind}}
\def\Res{\mbox{\rm Res}}
\def\Coind{\mbox{\rm CoInd}}
\def\|{\, | \, }
\def\bra{\langle}
\def\ket{\rangle}

\ncm{\rarr}[1]{\stackrel{#1}{\longrightarrow}}
\ncm{\larr}[1]{\stackrel{#1}{\longleftarrow}}

\def\eps{\varepsilon}

\def\du1{\hat 1}

\def\-1{_{(-1)}}
\def\0{_{(0)}}
\def\1{_{(1)}}
\def\2{_{(2)}}
\def\3{_{(3)}}

\def\du1{\hat 1}

\renewcommand{\thefootnote}{\fnsymbol{footnote}}
\title{\MakeUppercase{Subgroups of depth three and more}}

\author{Sebastian Burciu$^{1}$ and Lars Kadison$^{2}$}
\date{}
\renewcommand{\date}{\vspace{-5mm}}
\begin{document}
\maketitle \vspace*{-3mm}\relax
\renewcommand{\thefootnote}{\arabic{footnote}}

\noindent \textit{\small $^1$ Inst.\ of Math.\ ``Simion Stoilow" of the Romanian Academy
P.O. Box 1-764,\\ Bucharest, Romania, RO-014700\\
e-mail: sebastian.burciu@imar.ro}\\
\noindent \textit{\small $^2$ Department of Mathematics, University of Pennsylvania,
David Rittenhouse Lab, 209 S. 33rd St. \\
Philadelphia, PA 19104\\ e-mail: lkadison@math.upenn.edu}

\begin{abstract}
\noindent A subalgebra pair of semisimple complex algebras $B \subseteq A$ with inclusion matrix $M$ is depth two iff  $MM^t M \leq nM$ for some positive integer
$n$ and all corresponding entries.  If $A$ and $B$ are the
group algebras of finite group-subgroup pair $H < G$,
 the induction-restriction table equals
$M$ and $S = MM^t$ satisfies $S^2 \leq nS$ iff the subgroup
$H$ is depth three in $G$; similarly
depth $n > 3$ by successive right multiplications of this inequality with alternately $M$ and $M^t$.
We show that a Frobenius complement in a Frobenius group
is a nontrivial class of examples of depth three subgroups.  A tower of Hopf algebras $A \supseteq B \supseteq C$ is shown to be depth-3 if $C \subset \mtr{core}(B)$; and this is 
also a necessary condition if $A$, $B$ and $C$
are group algebras.
\end{abstract}

\section*{Introduction}
Induction of characters from a subgroup to a group is a useful technique for completing character tables \cite{I}
found by nineteenth century algebraists.  At about the same time, Frobenius discovered reciprocity, which in modern terms
states that induction is naturally isomorphic to coinduction of $G$-modules, either forming an adjoint pair with the restriction functor,
and applies to any Frobenius extension of algebras.

Finite index subfactors are a certain type of Frobenius extension, where an analytic notion of finite depth was discovered in connection with classification, with depth two being part of a remarkable type of Galois theory of paragroups.  The notion of finite depth was eventually made algebraic
and applied to Frobenius extensions; later, depth two and its Galois
theory of quantum groupoids and Hopf algebroids were exposed in simplest terms for ring extensions (see \cite{Sz} for an application to J.~Roberts field algebra construction \cite{R}).

It was noted in \cite{KK} that the  notion of depth two
applies to characters of a finite group and subgroup pair via complex group algebras:
a subgroup is depth two if  no new constituents arise
when inducing-restricting-inducing a character as compared with  inducing just
one time. By means of general theory in one direction and Mackey theory
in the other,  depth two subgroup is shown to be precisely a normal subgroup \cite{KK}.  (A similar statement is true for semisimple Hopf $\C$-subalgebras \cite{BK}.)
In this paper we generalize this approach to depth two subgroup to a semisimple
subalgebra pair, giving a condition in terms of inclusion matrix
\cite{GHJ},
which is the same as a induction-restriction table \cite{AB} up
to a permutation change of basis. The depth two condition is essentially that
the cube of the inclusion matrix is less entrywise than a multiple
of the inclusion matrix, noted more precisely in the abstract
and Proposition~\ref{prop-d2inclusion} below.

In \cite{LK2008} it was shown that finite depth Frobenius extension
has a simplified definition in terms of a generalization of
depth two to a tower of three algebras in the Jones tower. In this
paper we extend a particular case of an embedding theorem in \cite{LK2008} to characterization of certain finite depth separable
Frobenius extension in terms of depth two extension in Jones tower
(see Theorems~\ref{th-d3} and~\ref{th-hd} below).  Then one may check that
a subgroup is depth three or more by comparing cube of symmetric
matrix $\mathcal{S}$ of inner products of induced irreducible characters with multiples of $\mathcal{S}$ (see Prop.~\ref{prop-cube}). In somewhat the same spirit,   Corollary~\ref{cor-nonew} below implies that a subgroup is depth three
if no new constituents arise from applying restriction-induction
one extra time to a character.

Although amusing to test for depth three property from character tables of groups and non-normal subgroup,
 it is not clear from this definition
what precisely a depth three subgroup is.  A number of proposals to remedy this are given below:  depth three quasi-bases are given
in Theorem~\ref{th-d2qb},  a characterization of certain depth three Frobenius
extension in terms of similar bimodules, tensor-square and overalgebra in Theorem~\ref{th-tensorsquare}, and a class of examples in
 Section~3, a Frobenius group and its Frobenius complement.  Even the notion of depth-3 tower of algebras may be
viewed as an alternative to defining finite depth in terms of iterated endomorphism algebra  extensions (perhaps applied instead to an  iteration of another useful construction). Depth-3 towers of finite group algebras are completely classified in Theorem \ref{d3grptw} following the spirit of \cite{LK2008}. Depth-3 towers of Hopf algebras are also considered at the end of the second section. A tower of Hopf algebras $A \supseteq B \supseteq C$ is depth-3 if $C \subseteq \mtr{core}(B)$ (see subsection \ref{hophtw} for the definition of the core of a Hopf subalgebra). Using then notion of kernel of a module introduced in \cite{B} we formulate a conjecture on the $core$ of a Hopf subalgebra. This conjecture would imply that the condition $C \subseteq \mtr{core}(B)$ is also a necessary condition for the Hopf algebra tower $A \supseteq B \supseteq C$ to be depth-3 (which is true for group algebras
by the Theorem~\ref{d3grptw} below). Although our algebras are often over the complex numbers, the paper is hopefully written in
a change-of-characteristic-friendly way.

\section{Preliminaries on depth two extensions}\selabel{1}

All algebras in this paper are associative algebras (not necessarily commutative) over a field
$k$.  Given an $(A,A)$-bimodule $M$, we let
$M^A$ denote the $A$-central elements
$\{ m \in M \| \forall \, a \in A, am = ma \}$.

Two $r \times s$  matrices $M$ and $N$ of non-negative integers satisfy $M \leq N$ if
each of the coefficients $m_{ij} \leq n_{ij}$: this property is independent of permutation of bases. Note that if $X$ is a third $q \times r$
matrix of non-negative integers, then $XM \leq XN$; if $X$ is $s \times q$, then $MX \leq NX$. We say $M$ is strictly positive if
all entries $m_{ij} > 0$.

\subsection{Frobenius extensions}
\label{subsec-fe} A Frobenius extension $A \| B$ is an extension of associative algebras where the
natural bimodule ${}_BA_A$ is isomorphic
to the $(B,A)$-bimodule $\Hom (A_B, B_B)$
(of right $B$-module homomorphisms) given
by $(b\cdot f \cdot a)(x) =bf(ax)$
for $a,x \in A, b \in B, f\in \Hom (A_B, B_B)$.  This is equivalent to the
existence of a mapping $F \in \Hom ({}_BA_B, {}_BB_B)$ with dual bases $\{ x_i \}_{i=1}^n$ and $\{ y_i \}_{i=1}^n$
such that $\sum_{i=1}^n F(ax_i)y_i = a$
and $\sum_{i=1}^n x_iF(y_i a) = a$ for
all $a \in A$: we call the data system $F$ a Frobenius homomorphism with
dual bases $\{ x_i \}$, $\{ y_i \}$.

For example, a group algebra $A = k[G]$
is a Frobenius extension of any subgroup
algebra $B=k[H]$, where $H \leq G$ is a subgroup of finite index $[G: H] = n$.
For if $\{ g_i \}_{i=1}^n$ denotes left
coset representatives of $H$ in $G$,
where $g_1 = 1_G$,
a Frobenius system is given by $x_i = g_i^{-1}$, $y_i = g_i$ with  bimodule projection (then also split extension)
given by ($n_{hg_i} \in k$)
\begin{equation}
F(\sum_{i=1}^n \sum_{h \in H} n_{hg_i}  hg_i) = \sum_{h \in H} n_h h,
\end{equation}
a routine exercise.

A Frobenius extension $A \| B$ enjoys isomorphic tensor-square and endomorphism ring as $(A,A)$-bimodules.  We note that
$A \otimes_B A \cong \End A_B$ via
$x \otimes_B y \mapsto \lambda(x) \circ F \circ \lambda(y)$.  Also $A \otimes_B A \cong \End {}_BA$ via $x \otimes y \mapsto
\rho(y) \circ F \circ \rho(x)$  \cite{NEFE}.  Composing
the two isomorphisms we obtain an anti-isomorphism  $\End A_B \to \End {}_BA$
given by
$ f \mapsto \sum_i F(-f(x_i))y_i$,
which restricts to an anti-automorphism
on the subring $\End {}_BA_B$, and plays
the role of antipode in
case of depth two Frobenius extension
defined below.

\subsection{Separable extensions} If the characteristic of the ground field
$k$ is coprime to $[G: H] = n$, then
the extension of group algebras $A \| B$ noted above is a separable extension:
i.e., the multiplication map $\mu: A \otimes_B A \rightarrow A$ is a split $(A,A)$-epimorphism.  The image of $1_A$
under a section $A \to A \otimes_B A$
is a separability element $e= \sum_{i=1}^n e_i \otimes_B f_i$  satisfying $ae= ea$
for all $a \in A$ and $\mu(e) = \sum_{i=1}^n e_if_i = 1_A$, which characterizes separable extension.  Notice
that \begin{equation}
\frac{1}{[G: H]} \sum_{i=1}^n g_i^{-1} \otimes_B g_i
\end{equation}
is a separability element for the group
algebras $A$ over $B$.

In the situation that $C \supseteq A \supseteq B$ is a tower of algebras and $A \| B$ is a separable extension, the canonical epi $C \otimes_B C \to C \otimes_A C$ given by $c_1 \otimes_B c_2 \mapsto c_1 \otimes_A c_2$ splits.
A section for this mapping is of course
given by $c_1 \otimes_A c_2 \mapsto
\sum_{i=1}^n c_1 e_i \otimes_B f_i c_2$.

\subsection{Depth-3 towers of algebras} A  tower of three algebras
$A \supseteq B \supseteq C$, where $C$ is
a unital subalgebra of $B$ which is in turn unital subalgebra of $A$,
is said to be \textit{right depth-3},
or right d-3, if
there is a complementary $(A,C)$-bimodule
$P$ and $n \in \N$ such that
\begin{equation}
\label{eq: D2}
A \otimes_B A \, \oplus P \cong A^n
\end{equation}
 as natural $(A,C)$-bimodules.  Equivalently, there is a split $(A,C)$-bimodule epimorphism from a finite
direct sum of $A$ with itself to $A \otimes_B A$  ($P$ is the kernel of such an epi).

Left d-3 towers are defined oppositely,
so that $A \supseteq B \supseteq C$ is
left d-3 iff the tower of opposite algebras $A^{\rm op} \supseteq B^{\rm op} \supseteq C^{\rm op}$
is right d-3.  It has been noted in \cite{LK2008, CK} that
if $A \| B$ is a Frobenius, or quasi-Frobenius (QF, where isomorphisms above are replaced by similarity of bimodules) extension, then left d-3
is equivalent with right d-3 extension.

\subsubsection{Depth-3 towers of semisimple algebras}\label{d3ss}
Suppose a tower $A \supseteq B \supseteq C$ of semisimple finite dimensional $k$-algebras is right d-3.
Tensoring eq.~(\ref{eq: D2}) by $-\ot_CM$, we obtain the following inequality:   $$<M\uw^A_C\dw^A_B\uw_B^A,\;Q>\ \leq \ n<M\uw^A_C,\;Q>$$
which holds for any simple $C$-module $M$ and any simple $A$-module $Q$.

Using this relation a necessary and sufficient condition for a tower of groups to be depth-3 will be given in the next theorem.
For $H$ a subgroup of $G$ let $$\mathrm{core}_G(H)=\cap_{g \in G}\;^gH$$ be the largest subgroup of $H$ which is normal in $G$. (Here $^gH=gHg^{-1}$.)

Let $G \supseteq N \supseteq H$ be a tower of groups. Since the normal closure $H^G$ is the subgroup of $G$ generated by the elements $ghg^{-1}$ with $g \in G$ and $h \in H$ note that $H \subseteq \mathrm{core}_G(N)$ if and only if $H^G \subseteq N$.

\bn{theorem}\label{d3grptw}
A tower $G \supseteq N \supseteq H$ of groups is depth-3 if and only if $H \subset \mathrm{core}_G(N)$.
\end{theorem}

\bn{proof}

If $H \subset \mathrm{core}_G(N)$ then $H^G \subseteq N$ and the proof of Theorem 3.1 from \cite{LK2008} applies.

Suppose now that the tower is depth-3. The above argument for the tower $kG \supseteq kN \supseteq kH$ of semisimple algebras implies that there is $n \in \mathbb{N}$ such that

$$<\al\uw^G_H\dw^G_N\uw_N^G,\;\mu>\leq n<\al\uw^G_H,\;\mu>$$ for any characters $\al \in \mtr{Irr}(H)$ and $\mu \in \mtr{Irr}(G)$.

Put $\mu={1_{ _G}}$, the trivial character in the above inequality. Since $<\al\uw^G_H,\;{1_{ _G}}>=<\al,\;{1_{ _H}}>$ it follows that $<\al\uw^G_H\dw^G_N\uw_N^G,\;{1_{ _G}}>=0$ if $\al \neq {1_{ _H}}$. By Frobenius reciprocity this implies that $<\al\uw^G_H\dw^G_N,\;{1_{ _N}}>=0$ if $\al \neq {1_{ _H}}$.

On the other hand applying Mackey's theorem one has:
\bn{eqnarray*}
0=<\al\uw^G_H\dw^G_N,\;{1_{ _N}}> \! & = &  \! \sum_{NgH \in N\backslash G/H}<\;^g\al\dw_{N\cap \;^gH}^{^gH}\uw^N_{N\cap \;^gH},\;{1_{ _N}}>
 \\& = & \sum_{NgH \in N\backslash G/H}<\;^g\al\dw_{N\cap \;^gH}^{^gH},\;1_{ _{N\cap \;^gH}}>
\\& = & \sum_{NgH \in N\backslash G/H}<\al\dw_{^{g^{-1}}N\cap \;H}^{H},\;1_{ _{^{g^{-1}}N\cap \;H}}>
\\& = & \sum_{NgH \in N\backslash G/H}<\al,\;1_{ _{^{g^{-1}}N\cap \;H}}\uw_{^{g^{-1}}N\cap \;H}^{H}>
\end{eqnarray*}
On the other hand using Frobenius reciprocity again one has $$<{1_{ _H}},\;1_{ _{^{g^{-1}}N\cap \;H}}\uw_{^{g^{-1}}N\cap \;H}^{H}>=<1_{ _{^{g^{-1}}N\cap \;H}},\;1_{ _{^{g^{-1}}N\cap \;H}}>=1$$

Thus $$1_{ _{^{{g^{-1}}}N\cap \;H}}\uw_{^{g^{{-1}}N}\cap \;H}^{H}={1_{ _H}}$$ which implies that $H=\;^{g^{-1}}N\cap \;H$ or $H \subset \;^{g^{{-1}}}N=g^{-1}Ng$. Thus $H \subset \mathrm{core}_G(N)$.
\end{proof}

\subsection{Depth two algebra extensions} An algebra extension $A \supseteq B$ is defined
to be \textit{right depth two} (equivalently,
subalgebra $B \subseteq A$ is rD2) if the partially
trivial tower $A \supseteq B \supseteq B$
is right d-3; similarly we define
left D2 in terms of partially trivial
left d-3 tower.

It is obvious that a finite dimensional algebra $A$ is a depth two
extension of its unit subalgebra $B = k1_A$:  if $\dim_k A = n$,
then of course ${}_AA \otimes_k A \cong {}_AA^n$.  Similarly,  we
may show that if $C$ is a finite dimensional dimensional algebra,   the tensor algebra
$A = C \otimes B$ is a depth two extension of its subalgebra
$B = 1_C \otimes B$.

The main examples in the literature of depth two extension are Hopf-Galois extensions as well as its classical, weakened and pseudo- variants.

The defining Condition~(\ref{eq: D2}),
with $B = C$,
for right depth two extension is similar to
the characterization of projective module
as isomorphic to a direct summand of a free
module.  Like the derivation of projective bases for a projective module, we may derive from this condition right D2 quasi-bases for the right D2 extension $A \| B$ as follows.  For any ring extension,
using the hom-tensor relation, note that
$\Hom ({}_AA \otimes_B A_B, {}_AA_B) \cong
\End {}_BA_B$.  By  evaluation at $1_A$
note that $\Hom ({}_AA_B, {}_AA \otimes_B A_B) \cong (A \otimes_B A)^B$.

Then the split epi from $\pi: A^n \to A \otimes_B A$ satisfies an equation $\pi \circ \sigma = \id_{A \otimes_B A}$. We have $n$ standard split epis
$A^n \to A$, which compose with $\pi$ and $\sigma$ to give the equation $\sum_{i=1}^n f_i \circ g_i = \id_{A \otimes_B A}$,
where $f_i \in \Hom (A, A \otimes_B A)$
and $g_i \in \Hom (A \otimes_B A, A)$,
to which we apply the simplifications noted above.  Suppose $f_i \mapsto u_i \in (A \otimes_B A)^B$, while $g_i \mapsto \gamma_i \in \End {}_BA_B$ for each
$i = 1,\cdots,n$.
As a consequence, we obtain for
any $x, y \in A$ the identity
\begin{equation}
\label{eq: d2qb}
x \otimes_B y = \sum_{i=1}^n x \gamma_i(y)u_i
\end{equation}

Note that an extension $A \| B$ having elements $u_i \in (A \otimes_B A)^B$ and endomorphisms
$\gamma_i \in \End {}_BA_B$ satisfying this identity, eq.~(\ref{eq: d2qb}),
also implies that $A \| B$ is right D2, since $A^n \to A \otimes_B A$
given by $(a_1,\ldots,a_n) \mapsto \sum_i a_i u_i$ is an $(A,B)$-epimorphism
with section given by $x \otimes_B y \mapsto (x \gamma_1(y),\ldots,x\gamma_n(y))$.

For example, a normal subgroup $N$ of index $n$ in any group $G$
(over any ground ring)
is depth two with D2 quasi-bases given by $u_i = g_i^{-1} \otimes g_i$
and $\gamma_i(g) = F(g g_i^{-1})g_i$ for coset representatives
$\{ g_1 = e, g_2,\ldots,g_n \}$.

\subsection{When inclusion matrix is depth two}
Let the ground field $k = \C$ be the complex numbers when we consider
semisimple algebras, which consequently become multi-matrix algebras
(or split semisimple algebras).
Suppose $B \subseteq A$ is a subalgebra pair of semisimple algebras.
As one constructs an induction-restriction table for a subgroup
$H$ in a finite group $G$ \cite[p. 166]{AB}, we briefly review the  procedure
for generalizing to any pair of semisimple algebras (such as
finite dimensional complex group algebras).  Label the simples
of $A$ by $V_1, \ldots, V_s$ and the simple modules of $B$ by
$W_1, \ldots, W_r$.  To obtain the $i$'th column restrict the $i$'th
simple $A$-module $V_i$ to a $B$-module and express in terms of direct
sum of simples
\begin{equation}
V_j\!\downarrow_B \cong \oplus_{i = 1}^r m_{ij}W_i
\end{equation}
We let $M $ be the $r \times s$-matrix, or table, with entries $m_{ij}$:  $M = (m_{ij})$. By a well-known generalization of Frobenius reciprocity, the rows give induction of the $B$-simples:
\begin{equation}
W_i \! \uparrow^A = W_i^A = \oplus_{j = 1}^s m_{ij}V_j
\end{equation}
since $W_j^A = W_j \otimes_B A$ and $V_i \downarrow_B \cong \Hom (A_B,
V_i)$; i.e., if $[W_j^A, V_i]$ denotes the number of constituents
in $W_j^A$ isomorphic to $V_i$, Frobenius reciprocity is given
by
\begin{equation}
[W_i^A, V_j] = m_{ij} = [W_i, V_j\! \downarrow_B]
\end{equation}

The matrix $M$ is also known as the inclusion matrix of $B$
in $A$ \cite{GHJ}.

For example, the induction-restriction table (based
on Frobenius reciprocity $(\psi_i^G,\ \chi_j)_G \ = \ (\psi_i, \  \chi_j\downarrow_H)_H$) for the standard embedding of permutation
groups $S_2 \leq S_3$ is given by
$$
\begin{tabular}{l|lll}
$S_2 \leq S_3$ & $\chi_1 $ & $\chi_2$ & $\chi_3$ \\ \hline
$\psi_1 $  & 1 & 0 & 1 \\
$\psi_2$ & 0 & 1 & 1
\end{tabular}\ \ \ \
M = \left(
\begin{array}{ccc}
1 & 0 & 1 \\
0 & 1 & 1
\end{array}
\right) \ \ \ \
\begin{array}{rcccccccl}
 \stackrel{1}{\bullet} & & &&\stackrel{2}{\bullet} & & & &\stackrel{1}{\bullet} \\
&&&&&&&& \\
& \setminus & & / & & \setminus & &  /  & \\
&&&&&&&& \\
 & & \stackrel{\circ}{\mbox{\scriptsize 1}} & & & & \stackrel{\circ}{\mbox{\scriptsize 1}}  & &
\end{array}
$$
where $\psi_1 = 1_H$, $\chi_1= 1_G$ denote the trivial characters, $\psi_2$, $\chi_2$
the sign characters, and $\chi_3$ the two-dimensional
irreducible character of $S_3$.
Note too the inclusion diagram or Bratteli diagram, a bicolored weighted multigraph \cite{GHJ}.

For example, ${1_H}^G = \chi_1 + \chi_3$ and ${1_H}^G\downarrow_H =
2\cdot 1_H + \psi_2$.
Burciu \cite{B} notes that a subgroup $H$ is normal in $G$
if and only if ${1_H}^G\downarrow_H = [G:H] 1_H$.
In \cite{KK} it is established that the notion of depth two subalgebra for
subalgebra pair of complex group algebras is equivalent to the notion of normal subgroup.

\begin{proposition}
\label{prop-d2inclusion}
The inclusion matrix $M$ of a subalgebra pair of semisimple
complex algebras $B \subseteq A$ satisfies
\begin{equation}
MM^tM \leq n M
\end{equation}
for some positive integer $n$ if and only if $B$ is depth two subalgebra of $A$.
\end{proposition}
\begin{proof}
($\Leftarrow$) The depth two condition $A \otimes_B A \oplus P \cong A^n$ as
natural $B$-$A$-bimodules, becomes
\begin{equation}
[W_i^A\! \downarrow_B\uparrow^A, V_j] \leq n [W_i^A, V_j] = nm_{ij}
\end{equation}
for all $i= 1,\ldots,r$ and $j=1,\ldots,s$. But $W_i^A$ is given
by row $i$ of $M$, or $e_iM$, where $e_i$ denotes row matrix with all zeroes except $1$ in $i$'th column. Then $W_i^A \! \downarrow_B$
is given by $M (e_iM)^t = MM^te_i^t$. Finally $W_i^A\! \downarrow_B\uparrow^A$ is given by $(MM^te_i^t)^tM$, i.e. row $i$
of $MM^tM$.

($\Rightarrow$) If the inclusion matrix $M$ of semisimple subalgebra
pair $B \subseteq A$ satisfies $MM^tM \leq nM$ for some $n \in \Z_+$,
then $[\Ind^A_B \Res^A_B \Ind^A_B W_i, V_j] \leq n [\Ind^A_B W_i, V_j]$
for all $B$-simples $W_i$ and $A$-simples $V_j$ (fix these orderings).  
Via unique module decomposition into simples, we find a monic natural
transformation $\Ind^A_B \Res^A_B \Ind^A_B \stackrel{\cdot}{\hookrightarrow}
n\Ind^A_B$ from category $B$-\textrm{Mod} into $A$-\textrm{Mod}.
Now $B$, $A$ and so $ B^{\rm op}  \otimes A$ are separable $\C$-algebras,
so as in \cite[Theorem 2.1(6), pp. 3107-3108]{KK}, we apply the natural monic to
the right regular module $B_B$, apply the natural transformation property
to all left multiplications $\lambda_b$ ($b \in B$),
and note that 
$A \otimes_B A \hookrightarrow A^n$ splits by Maschke as $B$-$A$-bimodule monic.
Hence $A$ is depth two over its subalgebra $B$.  
\end{proof}

\subsection{Depth-3 tower of of Hopf algebras}\label{hophtw}
For $B \subset A$ an extension of finite dimensional Hopf algebras, define $\mtr{core}(B)$ to be the largest Hopf subalgebra of $B$ which is normal in $A$. It is easy to see that $\mtr{core}(B)$ always exists (see also \cite{B}). If $H\subset G$ is a group inclusion with $A=kG$ and $B=kH$ note that $\mtr{core}(B)=k\mathrm{core}_G(H)$.

\bn{theorem}\label{d3thopf}
Suppose that $A \supseteq B \supseteq C$ is a tower of semisimple Hopf algebras. If $C \subset \mtr{core}(B)$ then the tower is depth-3.
\end{theorem}

\bn{proof}
Since $\mtr{core}(B)$ is a normal Hopf subalgebra of $A$ it follows that the extension $\mtr{core}(B) \subset A$ is $D2$ and therefore $A\ot_{\mtr{core}(B)}A$ is a direct summand of the bimodule $_A(A^n)_{\mtr{core}(B)}$. Thus $A\ot_{\mtr{core}(B)}A$ is also a direct summand of the $A- C$  bimodule $_A(A^n)_C$ since $C \subset \mtr{core}(B)$.

Since $\mtr{core}(B)\subset B$ the canonical map $$A\ot_{\mtr{core}(B)}A \rightarrow A\ot_BA $$ is a surjective morphism of $A- A$-bimodules, in particular of $A-C$ bimodules. Since the category of $A \ot C^{op}$-modules is semisimple it follows that $A\ot_BA$ is a direct summand in $_A(A^n)_{C}$.
\end{proof}

\subsubsection{Kernel of a module}

Let $A$ be a semisimple Hopf algebra over an algebraically closed field $k$. Then $A$ is also cosemisimple and $S^2=\mtr{Id}$ (see \cite{Lard}). Let $\Lambda_{ _A}$ be the idempotent integral of $A$.
Denote by $\mtr{Irr}(A)$ the set of irreducible $A$-characters and let $C(A)$ be the character ring of $A$ with basis $\mtr{Irr}(A)$. There is an involution $``\;^*\;"$ on $C(A)$ determined by the antipode.

\bn{remark}\label{hopfgen} If $X \subset C(A^*)$ is closed under multiplication  then it generates a subbialgebra of $A$ denoted by $A_{ _{X}}$ \cite{NR'}. Moreover if $X$ is also closed under $``\;^*\;"$ it follows from the same paper that $A_{ _{X}}$ is a Hopf subalgebra. Since $A$ is finite dimensional any subbialgebra is a Hopf subalgebra and therefore any subset $X$ closed under multiplication is also closed under $``\;^*\;"$.
\end{remark}

Let $M$ be an $A$-module with character $\ch$.
Define $\mtr{ker}_M$ to be the set of simple subcoalgebras $C$ of $A$ such that $cm=\eps(c)m$ for all $c\in C$. It can be proven that the set $\mtr{ker}_M$ is closed under multiplication and $``\;^*\;"$ and therefore from \cite{NR'} it generates a Hopf subalgebra $A_{ _M}$ (or $A_{_ {\ch}}$) of $A$ \cite{B}. One has $A_{ _M}=\oplus_{C \in \mtr{ker}_M }C$.

\bn{remark}\label{rm}
1) $A_{ _{\ch}}$ is the largest subbialgebra $B$ of $A$ such that $\ch\dw^A_B=\ch(1)\eps_{ _B}$. Equivalently, $A_{ _{\ch}}$ is the largest subbialgebra $B$ of $A$ such that $AB^+A \subset \mtr{Ann}_{A}(M)$.

2) If $A=kG$ is a group algebra then $A_{_ {\ch}}=k[\mtr{ker}\;\ch]$ where $\mtr{ker}\;\ch$ is the kernel of the character $\ch$.

3) It is not known if $A_{ _{\ch}}$ is a normal Hopf subalgebra of $A$. In \cite{B} it was proven that $A_{\ch}$ is normal in $A$ if $\ch \in Z(A^*)$.

4) If $N$ is a submodule or a quotient of $M$ then clearly $A_{ _M} \subset A_{ _N}$ (since $A$ is semisimple).
\end{remark}

{\bf Notation:} If $B$ is a Hopf subalgebra of $A$ then we denote by  $\eps\uw^A_B$ the character $\eps_{ _B}\uw_B^A$.

\bn{proposition}
Suppose $B$ and $C$ are Hopf subalgebras of a finite dimensional semisimple Hopf algebra $A$. If
\bn{equation}\label{def}
_A(A\ot_B A)_C \oplus *\cong _AA^n\;_C
\end{equation}
as $A$-$C$-bimodules, then
$$C \subset  A_{ _{\eps_{ }\uw^A_{ _B}}}.$$
\end{proposition}

\bn{proof}
As in subsection \ref{d3ss} it follows that 
$$<M\uw^A_C\dw^A_B\uw_B^A,\;P>\ \leq \ n<M\uw^A_C,\;P>$$ for any simple left $C$-module $M$ and simple $A$-module $P$.

In terms of the characters this can be written as
\beq
{m_{ _A}}(\al\uw^A_C\dw^A_B\uw_B^A,\;\ch) \ \leq \ n\; {m_{ _A}}(\al\uw^A_C,\;\ch)
\eeq
for any irreducible character $\al$ of $C$ and any irreducible character $\ch$ of $A$. Here ${m_{ _A}}$ is the usual multiplication form on the character ring $C(A)$.
Put $\ch={\eps_{ _A}}$, the trivial $A$-character, in the above inequality. Since ${m_{ _A}}(\al\uw^A_C,\;{\eps_{ _A}})={m_{ _C}}(\al,\;{\eps_{ _C}})$ it follows that ${m_{ _A}}(\al\uw^A_C\dw^A_B\uw^A_B,\;{\eps_{ _A}})=0$ if $\al \neq {\eps_{ _C}}$. By Frobenius reciprocity this implies that ${m_{ _B}}(\al\uw^A_C\dw_B^A,\;{\eps_{ _B}})=0$ if $\al \neq {\eps_{ _C}}$. Adding over all irreducible characters $\al \in \mtr{Irr}(C)$ it follows that
\beq
{m_{ _B}}((\sum_{\al \in \mtr{Irr}(C)}\al(1)\al)\uw^A_C\dw_B^A,\;{\eps_{ _B}})={m_{ _B}}(\eps\uw^A_C\dw_B^A,\;{\eps_{ _B}})
\eeq

Since $\sum_{\al \in \mtr{Irr}(C)}\al(1)\al$ is the regular character of $C$ (see \cite{M}) it follows that $(\sum_{\al \in \mtr{Irr}(C)}\al(1)\al)\uw^A_C\dw_B^A$ is the regular character of $B$ multiplied by $\frac{|A|}{|C|}$. Thus ${m_{ _B}}(\eps\uw^A_C\dw_B^A,\;{\eps_{ _B}})=\frac{|A|}{|B|}$. Frobenius reciprocity implies that ${m_{ _C}}(\eps\uw^A_B\dw_C^A,\;{\eps_{ _C}})=\frac{|A|}{|B|}$. A dimension argument now shows that $\eps\uw^A_B\dw^A_C=\frac{|A|}{|B|}{\eps_{ _C}}$ and first item of Remark  \ref{rm} implies that $C \subset A_{ _{\eps_{ }\uw^A_{ _B}}}.$
\end{proof}

The above Proposition and Theorem \ref{d3thopf} suggest the following conjecture:

\bn{conj}  For any  Hopf subalgebra $B$ of a semisimple Hopf algebra $A$ one has:
\beq
\mtr{core}(B)= A_{ _{\eps_{ }\uw^A_{ _B}}}.
\eeq
\end{conj}

The next Proposition gives a description of $\mtr{core}(B)$ in terms of kernels and shows the inclusion $\mtr{core}(B) \subseteq A_{ _{\eps_{ }\uw^A_{ _B}}}$. In order to prove it we need the following lemmas.

\bn{lemma}\label{inc}
Let $K$ and $L$ be two Hopf subalgebras of a semisimple Hopf algebra $A$. If $\Lambda_{ _K}\Lambda_{ _L}=\Lambda_{ _L}$ then $K \subset L$.
\end{lemma}

\bn{proof}
By Corollary 2.5 of \cite{coset} there is a coset decomposition for $A$
\beq
A=\oplus_{C/\sim}C L.
\eeq
where $\sim$ is an equivalence relation on the set of simple subcoalgebras of $A$ given by $C \sim C'$ if and only if $CL=C'L$. In \cite{coset} this equivalence relation is denoted by $r^{A}_{ _{k,\;L}}$. The equality $\Lambda_{ _K}\Lambda_{ _L}=\Lambda_{ _L}$  shows that any subcoalgebra of $K$ is equivalent to $k1$ and therefore it is contained in $L$.
\end{proof}

\bn{lemma}
Suppose that $B$ is a Hopf subalgebra of a semisimple finite dimensional Hopf algebra.
Then $A_{ _{\eps_{ }\uw^A_{ _B}}}\subset B$. Equality holds if and only if $B$ is normal in $A$.
\end{lemma}

\bn{proof}
Let $K=A_{ _{\eps_{ }\uw^A_{ _B}}}$.
By the definition of $K$ it follows that $AK^+$ annihilates $A\ot_B k$. On the other hand $A\ot_Bk\cong A/AB^+$ as $A$-modules and therefore $AK^+ \subset AB^+$.  Thus $1- \Lambda_K \in AB^+$ which implies $\Lambda_{ _K}\Lambda_{ _B}=\Lambda_{ _B}$. The above Lemma implies that $K\subset B$. The second statement of the lemma is Corollary 2.5 from \cite{B}.
\end{proof}

\bn{proposition}\label{coredesc}
Suppose that $B$ is a Hopf subalgebra of a semisimple finite dimensional Hopf algebra $A$. Define inductively $$B_0=B,\;\; B_{s+1}=A_{ _{\eps\uw^A_{ _{B_s}}}}.$$ Then $$B_1 \supseteq B_2 \supseteq \cdots \supseteq B_s \supseteq B_{s+1}\supseteq \cdots.$$ If $B_s=B_{s+1}$ then $B_s =\mtr{core} (B)$.
\end{proposition}

\bn{proof}
The above proposition implies that $B_r \supseteq B_{r+1}$ for any $r$. Since $B$ is finite dimensional there is $s$ such that $B_s=B_{s+1}=B_{s+2}=\cdots$. Thus $B_s=A_{ _{\eps\uw^A_{ _{B_s}}}}$ and the above lemma implies that $B$ is normal in $A$. We have to show that $\mtr{core}(B)=B_s$. Suppose that $K$ is normal in $A$ and that $K \subseteq B$. It is enough to show $K \subseteq B_s$. Clearly $K \subseteq B_0$. If $K \subseteq B_i$ then there is a canonical surjection of $A$-modules $A/AK^+ \rightarrow A/A B_i^+$. Thus $A_{ _{\eps\uw^A_K}} \subset A_{ _{\eps\uw^A_{B_i}}}$ by the last item of Remark \ref{rm}. On the other hand $A_{ _{\eps\uw^A_K}}=K$ since $K$ is normal. Therefore $K \subseteq B_{i+1}$.
\end{proof}

\subsubsection{The Correspondent of conjugate Hopf subalgebras}
Let $A$ be a semisimple Hopf algebra over an algebraically closed field $k$
and let $\widehat{A^*}$ be the set of simple subcoalgebras of $A$. Since $A$ is cosemisimple note that $\widehat{A^*}$ can be identified with $\mtr{Irr}(A^*)$ \cite{L}.
Let $B$ be a Hopf subalgebra of $A$ and $C$ a simple subcoalgebra of $A$. Define
$$X_{_ {^CB}}=\{D \in \widehat{A^*}\;|\;dc{\Lam_{ _B}}=\eps(d)c{\Lam_{ _B}}\;\;\text{for all} \;\;c\in C, \;d \in D\;\}$$

\bn{proposition}
The set $X_{_ {^CB}}$ is closed under multiplication and $``\;^* \;"$ and it generates a Hopf subalgebra $^CB$ of $A$.
\end{proposition}

\bn{proof}
By Remark \ref{hopfgen} it is enough to show that the above set is closed under multiplication.
Suppose that $D$ and $D'$ are subcoalgebras in $X_{_ {^CB}}$. If $E$ is a simple subcoalgebra of $DD'$ then any $e \in E$ can be written as $\sum_{i=1}^sd_id'_i$ with $d_i \in D$ and $d'_i \in D'$. Then $ec{\Lam_{ _B}}=\eps(e)c{\Lam_{ _B}}$ which show that $E \in X_{_ {^CB}}$.
\end{proof}

{\bf Notation: }$^CB$ will also be denoted with $^cB$ if $c$ is the irreducible character of $A^*$  corresponding to $C$.

\bn{example}
\begin{rm} Let $A=kG$ and $B=kN$ where $N$ is a subgroup of $G$. The simple subcoalgebras of $A$ are $kg$ with $g \in G$ and $\Lam_{ _B}=\frac{1}{|N|}\sum_{n \in N}n$. Then $^gB=gBg^{-1}$ for all $g \in G$. Indeed $X_{ _{^gB}}=\{h \in G\;\;|\;hg{\Lam_{ _B}}=g{\Lam_{ _B}}\}=\{h \in G\;\;|\;hgN=gN\}=gNg^{-1}$
\end{rm}
\end{example}

\bn{proposition}
Let $B$ be a Hopf subalgebra of $A$ and $g \in G(A)$ be a grouplike element of $A$. Then $^gB=gBg^{-1}$.
\end{proposition}

\bn{proof}
First note that $^1B=B$. Clearly $B \subset \;^1B$. On the other hand the definition of $^1B$ implies that $\Lam_{ _{\;\;^1B}}\Lam_{ _B}=\Lam_{ _B}$. Then Lemma \ref{inc} implies $^1B\subset B$.

Let now $C$ be a simple subcoalgebra of $^gB$. Then $cg\Lam_{ _B}=\eps(c)g\Lam_{ _B}$ for all $c\in C$. Thus $g^{-1}cg{\Lam_{ _B}}=\eps(c){\Lam_{ _B}}$ which shows that $g^{-1}Cg\subset \;^1B=B$. Therefore $C\subset gBg^{-1}$ which shows that $^gB \subset gBg^{-1}$. A direct computations shows that $gBg^{-1}\subset ^gB $. Thus $^gB=gBg^{-1}$.

\end{proof}

\bn{proposition}
Let $B$ be a Hopf subalgebra of $A$. Then
$$A_{ _{\eps\uw^A_B}}=\cap_{C \in \widehat{A^*}}\;\;^CB.$$
\end{proposition}

\bn{proof} Recall the coset decomposition
\beq\label{cos}
A=\oplus_{C/\sim}CB.
\eeq
form Corollary 2.5 of \cite{coset}.
If $k$ is the trivial $B$-module then $$k\uw^A_B=\oplus_{C/\sim}CB\ot_Bk.$$ From the definition of $^CB$ it follows that $CB\ot_B k$ is trivial as left $^CB$-module. Therefore $\cap_{C \in \widehat{A^*}}\;\;^CB\subset A_{ _{\eps\uw^A_B}}$.

Note that $k\uw^A_B=A\ot_Bk\cong A\Lam_{ _B}$ as left $A$-modules via $a\Lam_{ _B} \mapsto a\ot_B \Lam_{ _B}$. The decomposition \ref{cos} implies that $CB\ot_Bk\cong C\Lam_{ _B}$ under the above isomorphism. Any simple subcoalgebra of $A_{ _{\eps\uw^A_B}}$ acts trivially on $k\uw^A_B$ and therefore on each $CB\ot_Bk$. This implies that any such coalgebra is contained in $^CB$. Thus $A_{ _{\eps\uw^A_B}} \subset {}^CB$ for any simple coalgebra $C \in \widehat{A^*}$.
\end{proof}

\bn{cor}
Let $B$ be a Hopf subalgebra of $A$. Then
$B$ is a normal Hopf subalgebra if and only if $$B=\cap_{C \in \widehat{A^*}}\, {}^CB.$$
\end{cor}

\bn{proof}
Since $A_{ _{\eps\uw^A_B}}=\cap_{C \in \widehat{A^*}}\, {}^CB$, this is Corollary 2.5 of \cite{B}.
\end{proof}

\bn{remark}
1)Theorem \ref{coredesc} implies that $\mtr{core}(B) \subset A_{ _{\eps\uw^A_B}}=\cap_{C \in \widehat{A^*}}\, {}^CB$. This can also be seen directly as follows. Fix $C \in  \widehat{A^*}$. For any $x \in \mtr{core}(B)$ and $c \in C$ one has that $xc\Lam_{ _B}=c_1(S(c_2)xc_3)\Lam_{ _B}=c_1\eps(x)\eps(c_2)\Lam_{ _B}=\eps(x)c\Lam_{ _B}$ since $\mtr{core}(B) $ is normal in $A$. Thus $\mtr{core}(B) \subset {}^CB$.

2) If $A_{ _{\ch}}$ is normal Hopf algebra for any $\ch \in \mtr{Irr}(A)$ then Proposition \ref{coredesc} implies the above conjecture on the core of a Hopf subalgebra.
\end{remark}

\section{Depth three Frobenius extension}

A Frobenius extension $A \| B$
is defined to be \textit{depth three} if
 the following tower of subalgebras in the endomorphism ring $E = \End A_B$ is right
or left
depth-3:
 via the algebra monomorphism,
left multiplication $\lambda: A \into E$
given by $\lambda(a)(x) = ax$ ($x,a \in A$) we obtain the (ascending) tower,
 $\lambda(B) \subseteq \lambda(A) \subseteq E$. By \cite[Theorem 3.1]{LK2008} the given
tower is left d-3 if and only if the tower is right d-3.

The definitions and first properties of
depth two and three extensions are introduced in detail in \cite{LK2008}.
There it is determined that a tower of
three group algebras corresponding to
the subgroup chain $G \geq H \geq K$
is depth-3 if the normal closure
$K^G$ (of $K$ in $G$) is contained
in $H$ (and shown above in Theorem~\ref{d3grptw} to be a characterization of depth-3
tower of finite groups).  In \cite{KK} it is shown
that, with $k = \C$ and $G$ a finite group,  the group algebra
$A$ of $G$ is depth two over subgroup algebra $B$ of $H$
if and only if $H$ is a normal subgroup of
$G$. This normality result for depth two
subalgebras is extended to semisimple Hopf algebras over an algebraically closed field of characteristic zero in \cite{BK}.

The following is a characterization of depth three for a separable, Frobenius extension in terms of the more familiar depth two property.
The following is true more generally for QF-extensions  \cite[Theorem 3.8]{CK}.

\begin{theorem}
\label{th-d3}
Suppose $A \| B$ is a separable extension and Frobenius extension.
Let $E$ denote $\End A_B$ and $\lambda: A \into E$ be understood
as the extension $E \| A$.
The $A \| B$ is depth three if and only if the composite extension
$E \| B$  is depth two.
\end{theorem}
\begin{proof}[Sketch of Proof] ($\Rightarrow$) This direction does not
apply separability. By the Frobenius
extension property, we  noted above that $E \cong A \otimes_B A$ as $(A,A)$-bimodules. Then $E \otimes_A E \otimes_A E \cong
E \otimes_B E$ as natural $(E,E)$-bimodules. By definition of
right D3 extension, $E \otimes_A E$ is isomorphic to direct summand
of $E^n$ as natural $(E,B)$-bimodules for some $n \in \N$, whence
$E \otimes_A E \otimes_A E \cong E \otimes_B E$ is $(E,B)$-isomorphic
to a direct summand of $E \otimes_A E^n$, which in turn is isomorphic
to a direct summand of $E^{n^2}$ by the right D3 property.  Hence
$E \| B$ is right D2, since $E \otimes_B E \oplus * \cong E^{n^2}$
as natural $(E,B)$-bimodules.

($\Leftarrow$) There is a split $(E,B)$-epimorphism from $E^n \to E \otimes_B E $ for some $n \in N$.  In addition, there is a split $(E,E)$-epimorphism
from $E \otimes_B E \to E \otimes_A E$ by the separability property
of the extension $A \| B$.  Composing the two split epis we obtain
a split epi $E^n \to E \otimes_A E$ showing $A \| B$ is right D3.
\end{proof}

 The proposition below has a proof useful to the exposition, although the result is improved somewhat in subsection~\ref{subsect-hd}.\begin{proposition}
\label{prop-cube}
Let $M$ be the inclusion matrix of a subalgebra pair of semisimple
complex algebras $B \subseteq A$, and $\mathcal{S} = MM^t$.
  The symmetric matrix
 $\mathcal{S}$ satisfies
\begin{equation}
\mathcal{S}^3 \leq n \mathcal{S}
\end{equation}
for some positive
integer $n$
if and only if $B$ is a depth three subalgebra of $A$.
\end{proposition}
\begin{proof}
Let $M_m(\C)= \End_{\C} A = \mathcal{E}$ where $m = \dim A$,
which contains both $A$ and $B$ via left regular representation.
It is shown in \cite[2.3.5]{GHJ} that the centralizers $\mathcal{E}^A \subseteq \mathcal{E}^B$ have transpose inclusion matrix; i.e.
inclusion matrix of $A \into \End A_B$ via $a \mapsto \lambda_a$
is $M^t$.  It is not hard to show from transitivity of induction
that matrix multiplication yields new inclusion matrix of two successive subalgebra pairs.  Hence, inclusion matrix of
$B \into E$ via $b \mapsto \lambda_b$ ($b \in B$) is
given by $MM^t$.

 The algebra $A$ is separable, whence separable extension over $B$.  The extension $A \supseteq B$ is a split Frobenius extension
by application of \cite[Goodman-De la Harpe-Jones, ch.\ 2]{GHJ}, very faithful
conditional expectations. Then $A_B$ is a progenerator since $B$ is semisimple and $B \into A$ is split $B$-module monic, so $E$ and $B$ are Morita equivalent semisimple algebras.   By the theorem above,
$B \subseteq A$ is depth three iff $B \into E$ is depth two,
and we may apply Proposition~\ref{prop-d2inclusion} to the composite
inclusion matrix $\mathcal{S} = MM^t$.
\end{proof}

In general
for any subgroup $H$ in finite group $G$ with inclusion matrix $M$,
if the irreducible characters of $H$ are given by $\{ \psi_1, \ldots,\psi_r \} = \mbox{\rm Irr}(H)$, note that the matrix
$\mathcal{S} = MM^t$ is given by
\begin{equation}
\label{eq: sformula}
\mathcal{S} = \left(
\begin{array}{ccc}
\bra \psi_1^G | \psi_1^G \ket & \ldots & \bra \psi_1^G | \psi_r^G \ket \\
\ldots & \ldots & \ldots \\
\bra \psi_r^G | \psi_1^G \ket & \ldots & \bra \psi_r^G | \psi_r^G \ket
\end{array}
\right).
\end{equation}

For example, we revisit the inclusion $S_2 < S_3$ analyzed above.
Note that
\begin{equation}
\label{eq: s-matrix}
\mathcal{S} = MM^t = \left(
\begin{array}{cc}
2 & 1 \\
1 & 2
\end{array}
\right).
\end{equation}
Since $\mathcal{S}$ is strictly positive (i.e. has only positive whole number entries), it is clear that there is positive integer $n$
such that $\mathcal{S}^3 \leq n\mathcal{S}$.

The notation in the proposition above with finite dimensional complex group algebras $B = \C [H]$ and $A = \C [G]$ is continued in the next corollary:

\begin{cor}
\label{cor-pos}
The subgroup $H$ is depth three in $G$ if the matrix $\mathcal{S}$ is strictly positive.
\end{cor}

Another example: the standard inclusion of full permutation group algebras
$B= \C[S_3] \into \C[S_4] = A$ has inclusion matrix
(computed from character tables in e.g.\ \cite{FH}) and symmetric matrix:
$$
M = \left(
\begin{array}{ccccc}
1 & 0 & 0 & 1 & 0 \\
0 & 1 & 0 & 0 & 1  \\
0 & 0 & 1 & 1 & 1
\end{array}
\right) \ \
\mathcal{S} = \left(
\begin{array}{ccc}
2 & 0 & 1 \\
0 & 2 & 1 \\
1 & 1 & 3
\end{array}
\right)  \ \
S^3 = \left(
    \begin{array}{ccc}
15 & 7 & 21 \\
7 & 15 & 21 \\
21 & 21 & 43
\end{array}
\right)
$$

It is clear that there is no positive integer $n$ for which
$\mathcal{S}^3 \leq n \mathcal{S}$, since $\mathcal{S}$ has zero
entries but $\mathcal{S}^3$ is strictly positive.  We conclude
that
$S_3$ is not a depth three subgroup of $S_4$. (Using the next theorem one computes that $S_3$ is a depth five
subgroup of $S_4$.)
\subsection{Higher depth}\label{subsect-hd} Recall from \cite{LK2008} that depth $n > 2$
is defined as follows.  Begin with a Frobenius extension (or
QF extension \cite{CK}) $B = A_{-1} \subseteq A = A_0$.  Let
$A_1 = \End A_B$ and inductively $A_n = \End (A_{n-1})_{A_{n-2}}$.
By the Frobenius hypothesis and its endomorphism ring theorem, $A_n \cong A \otimes_B \cdots \otimes_B A$ ($n+1$ times $A$). Embedding
$A_n \into A_{n+1}$ via left regular representation $\lambda$, we obtain a Jones
tower of algebras,
$$ B \into A \into A_1\into \cdots \into A_n \into A_{n+1} \into \cdots $$  The subalgebra $B$ in $A$ is depth $n$ if $A_{n-2} \supseteq A_{n-3} \supseteq B$ is a depth-3 tower defined above;
infinite depth if there is no such positive integer $n$.   Of course, this agrees with the definition of depth three subalgebra above.
If $B$ and $A$ are semisimple complex algebras, $A \supseteq B$
becomes a split, separable Frobenius extension via the construction
of a very faithful conditional expection \cite{GHJ}.  This type
of extension has an endomorphism ring theorem \cite {NEFE}, and enjoys transitivity, so that all extensions in this Jones tower are split,
separable Frobenius extensions, and all algebras are semisimple by Morita's theorem
(or Serre's theorem on global dimension). Indeed, all the odd $A_n$'s
are Morita equivalent to $B$, while all the even $A_n$'s are Morita
equivalent to $A$.
The proof of the lemma below is similar to that of  Prop.~\ref{prop-d2inclusion} and therefore omitted.
(One notes that $\Ind_C^A \cong \Ind^A_B \Ind^B_C$ is given by the rows of  matrix $NM$.)
\begin{lemma}
Suppose $C \subseteq B \subseteq A$ is a tower of semisimple algebras
with inclusion matrices $N$ and $M$ respectively.  Then the tower is
depth-3 if and only if there is a positive integer $n$ such that
\begin{equation}
NMM^tM \leq n NM
\end{equation}
\end{lemma}
Notice that Prop.~\ref{prop-d2inclusion} follows from letting $B = C$
and $N$ equal the identity matrix of rank $\dim Z(B)$.  Conversely,
if $A \supseteq B$ is depth two, and $C$ any subalgebra of $B$, then
the lemma follows in this special case from Prop.~\ref{prop-d2inclusion} by multiplying the
inequality there from the left by the inclusion matrix $N$ of $C \subseteq B$.

Let $n, m$ and $q$ denote positive integers below.
\begin{theorem}
\label{th-hd}
Suppose $B \subseteq A$ is  a subalgebra pair of semisimple algebras. Let $M$ be the inclusion matrix and $\mathcal{S} = MM^t$. If $n = 2m + 1$ then $A \supseteq B$ is depth $n$ if and only if $\mathcal{S}^{m+1} \leq q\mathcal{S}^m $ for some $q$. If $n = 2m$, then $A \supseteq B$ is depth $n$ if and only if
$\mathcal{S}^m M \leq q\mathcal{S}^{m-1}M$ for some $q$.
\end{theorem}
\begin{proof}
The proof follows from noting that if $M$ is the inclusion matrix
of $B \subseteq A$, then $M^t$ is the inclusion matrix of $A \into
A_1$, and $\mathcal{S}$ is the inclusion matrix of their composite
$B \into A_1$.  The proof now follows from applying the last lemma
to the depth-3 tower $B \into A_{n-3} \into A_{n-2}$ in the even
and odd case.
\end{proof}
It is worth emphasizing that a depth $n$ algebra extension is also depth $n+1$ (so one might denote this as depth $\geq n$); in the special case of the theorem, this is seen by
multiplying the given inequality from the right by the inclusion matrix $M$ or $M^t$. Of course one should strive to use the least depth to one's knowledge.
Let $m$ be a positive integer and $G$ a finite group
 in the next result on  subgroups of finite depth.
\begin{cor}
\label{cor-dn}
Suppose $H < G$ is a subgroup with symmetric matrix $\mathcal{S}$.  If $\mathcal{S}^m$
is a strictly positive matrix, then $H$ is a subgroup of depth
$2m + 1$ in $G$.
\end{cor}
\begin{proof} Applying the theorem we see
$\mathcal{S}^{m+1} \leq q\mathcal{S}^m$ for some positive integer $q$
since $\mathcal{S}^m$ is a strictly positive matrix.
\end{proof}

  For example, while $S_3 < S_4$ is not D3 subgroup, we note
that $\mathcal{S}^2$ is already strictly positive order $3$ matrix,
whence it is a depth five subgroup (and it may be checked that it is not depth four).

As another example of a more cautionary note,  the symmetries of a square
$D_4$ in $S_4$ has zero entries in all powers of its order 5 matrix $\mathcal{S} = MM^t$.  However, one computes that $\mathcal{S}^2 M \leq 4 \mathcal{S}M$, so that $D_4$ is a depth four subgroup of $S_4$ according
to Theorem~\ref{th-hd}.

In a forthcoming paper it will be shown that after a permutation of the indices, the matrix $S$ can be written as a sum of diagonal blocks. Moreover there is $p>0$ such that the $p$-power of each diagonal block is a positive matrix.  Applying Theorem~\ref{th-hd} this implies that the extension $B \subseteq A$ is of depth  $\geq 2p+1$; in other words, all semisimple subalgebra pairs are of finite depth.

\subsection{Simplified condition for depth three}
Again let $A \supseteq B$ be an algebra extension.
In case $A_B$ is a generator, such as when the extension is free or right split, there is a particularly simplified condition for
when a Frobenius extension is depth three.
\begin{theorem}
\label{th-tensorsquare}
Suppose $A \supseteq B$ is a Frobenius extension where the natural
module $A_B$ is
a generator.  Then $A \supseteq B$ is depth three if and only if
there is a $B$-$B$-bimodule $P$ and positive integer $n$ such that
\begin{equation}
\label{eq: endresult}
{}_BA \otimes_B A_B \oplus P \cong {{}_BA_B}^n
\end{equation}
\end{theorem}
\begin{proof}
($\Rightarrow$) Let $E = \End A_B$.  By the Frobenius extension hypothesis on $A \supseteq B$, as $E$-$A$-bimodules $E \cong A \otimes_B A$ via the mapping in subsection~\ref{subsec-fe}.
Recall that $A \supseteq B$ is depth three if $B \subseteq A \into
E$ is depth three tower, i.e. ${}_EE \otimes_A E_B \oplus Q \cong {}_EE_B^n$ for some $E$-$B$-bimodule $Q$ and positive integer $n$.
Then by substitution
\begin{equation}
\label{eq: halfway}
{}_E A \otimes_B A \otimes_B A_B \oplus Q \cong {}_E A \otimes_B {A_B}^n.
\end{equation}
But $A_B$ is a progenerator by hypothesis, whence $B$ and $E$ are Morita equivalent algebras.   The context bimodule are ${}_B\Hom (A_B, B_B)_E$ (the right $B$-dual of $A$ denoted
by $(A_B)^*$) and ${}_EA_B$ with $B$-$B$-bimodule isomorphism $\Hom (A_B, B_B) \otimes_E A \stackrel{\cong}{\longrightarrow} B$
given by evaluation.
 Now tensor all components of eq.~(\ref{eq: halfway}) by ${}_B(A_B)^* \otimes_E -$
and cancel $B \otimes_B$ to obtain eq.~(\ref{eq: endresult}),
where of course $P = (A_B)^* \otimes_E Q$.

($\Leftarrow$) Tensor all components of eq.~(\ref{eq: endresult})
from the left  by
the natural bimodule ${}_EA_B$ given by $f \cdot a \cdot b = f(a)b$,
obtain eq.~(\ref{eq: halfway}), and reverse the argument above it.
This direction of proof does not make use of generator hypothesis.
\end{proof}

\begin{remark}
\begin{rm}
Since ${}_BA_B \oplus \Omega \cong {}_BA \otimes_B A_B$ is always the
case for some $B$-$B$-bimodule $\Omega$, the $B$-$B$-bimodules
$A \otimes_B A$ and $A$ are similar or H-equivalent under the conditions of the theorem:  thus their endomorphism algebras
are Morita equivalent.
By Theorem~\ref{th-d3}, left multiplication $B \into E$ is depth two.
There is a general Galois theory of depth two extensions which in this case specializes to total algebra $\End {}_BA \otimes_B A_B$ and base algebra $E^B \cong \End {}_BA_B$ as parts of a bialgebroid. It is interesting to note
that base and total algebras in this case are Morita equivalent.
\end{rm}
\end{remark}
Let $\Res = \Res^G_H$ denote restriction of $G$-modules to $H$-modules
in the corollary below, and $\Ind = \Ind^G_H$ denote induction of
$H$-modules to $G$-modules.
\begin{cor}
\label{cor-nonew}
A subgroup $H$ of a finite group $G$ is depth three if and only if
\begin{equation}
\label{eq: resind}
\bra \Res \Ind \Res \Ind \psi \| \chi \ket \leq n \bra \Res \Ind \psi \| \chi \ket
\end{equation}
for all irreducible characters $\psi, \chi$ of $H$.
\end{cor}
\begin{proof}
Note that the corresponding complex group algebras $A \supseteq B$
satisfy the conditions of the theorem.
One arrives at the condition on inner products of characters by
tensoring a simple $B$-module $V$ by the components in eq.~(\ref{eq: endresult}).  Of course, whatever simple $B$-module components of
${}_BA \otimes_B A \otimes_B V$ has, also ${}_BA^n \otimes_B V$ has.
\end{proof}

For example, from the character tables of the permutation groups
$S_4$ and $S_5$ \cite{FH} we compute the induction-restriction table by restricting irreducible characters on $S_5$, given below in matrix form
(with first column and row corresponding to trivial characters):

$$ M = \left( \begin{array}{ccccccc}
1 & 0 & 1 & 0 & 0 & 0 & 0 \\
0 & 1 & 0 & 1 & 0 & 0 & 0 \\
0 & 0 & 0 & 0 & 0 & 1 & 1 \\
0 & 0 & 1 & 0 & 1 & 1 & 0 \\
0 & 0 & 0 & 1 & 1 & 0 & 1
\end{array}
\right) $$

Let $\eta_i \in \mbox{\rm Irr}(H)$ and $\chi_j \in \mbox{\rm Irr}(G)$
 ($(i,j) \in \underline{5} \times \underline{7}$).  Then from row 1, $\eta_1^G = \chi_1 + \chi_3$,
$\eta_1^G\downarrow_H = 2\eta_1 + \eta_4$ and finally
$\eta_1^G\downarrow_H\uparrow^G\downarrow_H = 5\eta_1 +\eta_3 + 5\eta_4 + \eta_5$.  Note $\bra \eta_1^G\downarrow_H\uparrow^G\downarrow_H | \eta_3 \ket = 1 \not \leq n \bra  \eta_1^G\downarrow_H | \eta_3 \ket = 0$ for all
positive integers $n$,
whence $S_4$ is not a D3 subgroup in $S_5$.

Computing the $5 \times 5$ matrix $\mathcal{S} = MM^t$, we may compute
that the matrix $\mathcal{S}^3$ is strictly positive, so that $S_4$ is  a depth seven subgroup in $S_5$ by  Theorem~\ref{th-hd} and its corollary.
(Observing the pattern, we might conjecture at this point that the canonical subgroup $S_n < S_{n+1}$ has depth $2n-1$.)

\subsection{Depth three quasi-bases}

The condition~(\ref{eq: endresult}) for a depth three extension
has an interpretation in terms of split epis, including the canonical split epis of a product.  This should give us depth three condition
in terms of quasi-bases somewhat similar to dual bases for projective modules.  Meanwhile the Frobenius hypothesis on
extension $A \supseteq B$ is needed
to reduce the quasi-bases to simplest terms.  Suppose $F$
is a Frobenius homomorphism $A \rightarrow B$ with dual bases
$\{x_i \}$ and $\{ y_i \}$ in $A$.

\begin{theorem}
\label{th-d2qb}
Suppose $A \supseteq B$ is a Frobenius extension where $A_B$ is
a generator.  Then $A \supseteq B$ is a depth three extension
if and only if there are elements $u_i, t_i \in (A \otimes_B A \otimes_B A)^B$  such that for all $x,y \in A$,
\begin{equation}
\label{eq: d3qb}
x \otimes_B y = \sum_{i=1}^n t^1_i \otimes_B t^2_i F(t^3_i u^1_i F(u^2_i F(u^3_i x)y))
\end{equation}
where $u = u^1 \otimes u^2 \otimes u^3$ is Sweedler notation that
suppresses a possible summation over simple tensors.
\end{theorem}
\begin{proof}
($\Rightarrow$) First note from eq.~(\ref{eq: endresult}) that there are mappings \newline
 $f_i \in \Hom ({}_BA_B, {}_BA \otimes_B A_B)$ and $g_i \in \Hom ({}_BA \otimes_B A_B, {}_BA_B)$ such that $$\sum_{i=1}^n f_i \circ g_i = \id_{A \otimes_B A}.$$

Next recall that for any $B$-module $M$, $\Coind M \cong \Ind M$
for a Frobenius extension $A$ over $B$ \cite{NEFE};
I.e., there is a natural $A$-module isomorphism $\Hom (A_B, M_B) \cong
M \otimes_B A$ via $f \mapsto \sum f(x_i) \otimes y_i$
with inverse $m \otimes a \mapsto mF(a-)$.
Applied to $M = A \otimes_B A$, this restricts to
$\Hom ({}_BA_B, {}_BA \otimes_B A_B) \cong (A \otimes_B A \otimes_B A)^B$ via $f \mapsto \sum_i f(x_i) \otimes y_i$ with inverse
\begin{equation}
t \longmapsto t^1 \otimes t^2 F(t^3-).
\end{equation}

  Next  apply the hom-tensor relation and the Frobenius isomorphism
between endomorphism ring and tensor-square of extension:
$$\Hom({}_BA\otimes_B A_B, {}_BA_B) \cong \Hom (A_B, E_B)^B $$
$$\ \ \ \  \cong
\Hom ({}_BA_B, {}_BA \otimes_B A_B) \cong (A \otimes_B A \otimes_B A)^B.$$
Following the isomorphisms, the forward composite mapping is
given by $g \mapsto \sum_{i,j} g(x_i \otimes x_j) \otimes y_j \otimes y_i$ with inverse given by
\begin{equation}
u \longmapsto (x \otimes y \mapsto u^1F(u^2 F(u^3x)y))
\end{equation}
for all $u \in (A \otimes_B A \otimes_B A)^B$, $x,y \in A$.

Now suppose the mappings we begin with $f_i \mapsto t_i$ and $g_i \mapsto u_i$ in \newline
 $(A \otimes_B A \otimes_B A)^B$ via isomorphisms displayed above.  Then eq.~(\ref{eq: d3qb}) results.

($\Leftarrow$) Define a split $B$-$B$-bimodule epimorphism
$A^n \rightarrow A \otimes_B A$ by \newline
 $(a_1,\ldots,a_n) \mapsto
\sum_{i=1}^n t_i^1 \otimes t_i^2F(t_i^3 a_i)$ with section
$A \otimes_B A \rightarrow A^n$ given by
 $x \otimes y \mapsto$ \newline $(u_i^1F(u_i^2F(u_i^3 x)y))_{i=1,\ldots,n}$.
\end{proof}

For example, a left depth two quasi-bases $t_i \in (A \otimes_B A)^B$
and $\beta_i \in \End {}_BA_B$ for $A \supseteq B$
satisfy $x \otimes y = \sum_{i=1}^n t_i \beta_i(x)y$
for all $x,y \in A$.  If $A$ is Frobenius extension of $B$, then
$\End {}_BA_B \cong (A \otimes_B A)^B$ via $\alpha \mapsto
\sum_i \alpha(x_i) \otimes y_i$ with inverse $t \mapsto t^1 F(t^2 -)$.
Let $u_i \in (A \otimes_B A)^B$ satisfy $u_i^1 F(u_i^2-) = \beta_i$.
Then
\begin{equation}
\{ \sum_j t_i^1 \otimes_B t_i^2x_j \otimes_B y_j \}_{i=1,\ldots,n}
\ \ \{ \sum_j x_j \otimes_B y_ju_i^1 \otimes_B u_i^2 \}_{i=1,\ldots,n}
\end{equation}
are D3 quasi-bases, because
$$ \sum_{i,j,k} t_i^1 \otimes t_i^2 x_j F(y_j x_k F(y_k u_i^1 F(u_i^2 x)y ))
= \sum_{i,k} t_i^1 \otimes t_i^2 x_k F(y_k u_i^1 F(u_i^2x)y) $$
$$ =
\sum_i t_i^1 \otimes t_i^2 u_i^1 F(u_i^2 x)y = x \otimes y. $$


\section{Hall subgroup in Frobenius group is depth three}

A Frobenius group is a finite group $G$ with nontrivial normal subgroup $M$ (called the Frobenius kernel) which contains the centralizer of each of its nonzero elements: $C_G(\{x\}) \subseteq M$
for each $x \in M^*$ \cite{I, S}.  This is equivalent to $G$ having a Hall subgroup, or Frobenius complement, $H$ such that $G = MH$,
$M \cap H = \{e \}$, $H \cap H^x = \{ e \}$
where $H^x = x^{-1}Hx$ for any $x \in G - H$; in addition,
$ M = G - \cup_{x \in G} x^{-1}H^* x$.
The Hall subgroup $H$ is not normal in $G$ (and therefore not depth two in the terms of this paper).  We will see below that $H < G$ represents a nontrivial class of examples of depth three subgroup.

For example, the permutation group $S_3$ is a Frobenius group with kernel $M = \bra (1 2 3) \ket$ and  $H = S_2 = \bra(1 2)\ket$ or either of the two subgroups
$\bra (23) \ket$ or $\bra (13) \ket$ are Hall subgroups.

\begin{theorem}
Let $G$ be a Frobenius group with Hall subgroup $H$.  Then $H$
is depth three subgroup of $G$.
\end{theorem}
\begin{proof}
 From the defining condition~(\ref{eq: resind}), we easily find
a positive integer $n$ if $ \bra \Res \Ind \psi | \chi \ket > 0$
for all irreducible characters $\psi, \chi$ of $H$.   We compute using Mackey subgroup theorem \cite[p. 74]{I} and Frobenius reciprocity, where $T$ denotes a set of $n$ double coset representative $\{ e = g_1, g_2,\ldots,g_n \}$:
$$  \bra  \psi^G \!\downarrow_H |\, \chi \ket  = \sum_{t \in T} \bra (\psi^t \!\downarrow_{H^t \cap H})\!\uparrow^H, \chi \ket  =
\sum_{t \in T} \bra \psi^t \!\downarrow_{H^t \cap H} |\, \chi\!\downarrow_{H^t \cap H} \ket \geq n-1 $$
since $H^t \cap H = \{ e \}$ for each $t \neq g_1$.  Indeed it is easy to check that
$$\bra \psi^G\downarrow_H | \chi \ket  = (n-1) (\deg \psi)( \deg \chi)$$
if $\psi \neq \chi$ and equals $1 + (n-1)(\deg \psi)^2$ if
$\chi = \psi$.
\end{proof}

For example the subgroup $S_2$ in $S_3$ has two double coset reprentatives, both irreducible characters are linear, and the values
$\bra  \psi^G\downarrow_H | \chi \ket  = \bra \psi^G | \chi^G \ket$ are 1 on the off-diagonal and 2 on the diagonal,
 the coefficients of the matrix $\mathcal{S}$ in eq.~(\ref{eq: s-matrix}).
 The proof of the theorem also follows from eq.~(\ref{eq: sformula}), Corollary~\ref{cor-pos} and Mackey's theorem.

\subsection{Acknowledgements}
The second author is grateful to David Harbater and Gestur Olafsson for discussions related to this paper.

\end{document}